\documentclass[a4paper,11pt]{article}
\usepackage[top=20mm,bottom=32mm]{geometry}
\usepackage{amsmath}
\usepackage{amsfonts}
\usepackage{amsthm}
\usepackage{graphicx}
\usepackage{hyperref}
\DeclareMathOperator{\real}{Re}
\DeclareMathOperator{\diag}{diag}
\newcommand{\eps}{\varepsilon}
\newcommand{\ones}{\mathbf{1}}
\newcommand{\unit}[1]{I_{#1\times #1}}
\newtheorem{lemma}{Lemma}
\newtheorem{proposition}{Proposition}
\newtheorem{theorem}{Theorem}

\begin{document}

\title{Characterization of the Response Maps of
  Alternating-Current Networks}
\author{G\"unter Rote\\
\small Freie Universit\"at Berlin, Institut f\"ur Informatik\\
\small rote@inf.fu-berlin.de
}
\maketitle

\begin{abstract}
In an \emph{alternating-current network},
each edge has a complex \emph{conductance} with positive
real part.
The \emph{response map} is the linear map from the vector of
voltages at a subset of \emph{boundary nodes} to the vector of currents flowing into
the network through these nodes.
\par
We prove that the known necessary conditions for these response maps
are sufficient, and we %show how to
construct an
appropriate %alternating-current
network for a given response map.
\end{abstract}

{\small


}

\section{Problem Statement and Background}

An \emph{alternating-current network} is
 an undirected graph $G$ in which each edge
 $uw$ is assigned a \emph{conductance}
 $c_{uw}=c_{wu}\in \mathbb{C}$
 with positive real part: $\real c_{uw}>0$.
Such % alternating-current
 networks can
 model the physics of alternating current with a fixed frequency
%$\omega$
in an electric network of conductors, capacitors, and inductors
\cite[Section~2.4]{Prasolov}.
%A subset of
At least %$b\ge2$ 
$2$ of the nodes are designated as \emph{boundary nodes} (or \emph{terminals}).
%(or {terminals}).
Any remaining nodes % (if there are any)
are called \emph{interior nodes}.

A \emph{voltage} is
a complex-valued function $V_u$ on the set of nodes such that
the equilibrium condition
\begin{equation}
  \label{eq:equilibrium}
  \sum_{uw}c_{uw}(V_u-V_w)=0
\end{equation}
holds for each interior node $u$,
where the sum is taken over the edges $uw$ incident to%~ the node%
~$u$.
In a connected network, 
 the voltage is uniquely determined by its boundary
values
\cite[Section~5.1]{Prasolov}.
 The \emph{current flowing into the network} through a boundary
node $u$ is
\begin{equation}
  \label{eq:current}
I_u:=\sum_{uw}c_{uw}(V_u-V_w).   
\end{equation}
The
\emph{response map} is the linear map that takes the vector $(V_u)$ of
voltages at the boundary nodes to the vector $(I_u)$ of currents flowing
into the network through the boundary nodes.

Which linear maps are response maps of electrical networks?
%and it has even been solved under the constraint that the network $G$ is planar
%% Curtis, Ingerman, and Morrow,
%\cite{Curtis-Ingerman-Morrow,%}\cite{
%Curtis-Morrow00}.
% nicht genau: Circular-planar
This question
 has been posed as an open problem
%even without any constraints on the structure of $G$
 \cite[Problem~4.8]{Prasolov}, see
 also \cite[Questions 1 and 2]{r-open-15}. %The purpose of
 This note settles the problem:
Theorem~\ref{main} shows that the known necessary conditions are
sufficient.
Prasolov and Skopenkov~\cite[Section~4.2]{Prasolov} expressed hopes
that this conjectured solution of their question might allow %further
progress on tilings: deciding if a
 polygon can be tiled by rectangles with a given selection of possible
 aspect ratios.
%Problem 4.1
%Which polygons can be tiled by rectangles of given ratios c 1 , . . . , c n ?
%Problem 1.2. Which rectangles can be tiled by rectangles of given
%ratios c 1 , . . . , c n ?

% The more familiar
% \emph{direct-current} networks are the special case
% where
% all conductances are real
% and positive, because they are the inverses of the resistances.
% For
% this case, the characterization question has a trivial answer,
% see the remark after
% Theorem~\ref{main}.

The general \emph{electrical impedance tomography problem} is to
reconstruct the network from its response map. This problem is more
difficult and can only be solved when the structure of
the network is constrained, cf.~%, for example, planarity
\cite{Curtis-Ingerman-Morrow,%}\cite{
Curtis-Morrow00,deV1,deV2}.

% Consider the linear map
% $C^b\to C^b$
% which takes the vector of voltages

% ( U 1 ,..., U b )

% to the vector of
% incoming currents

% ( I 1 ,..., I b ) = ( $\sum$ n k = 1 I 1 k ,..., $\sum$ n k = 1 I bk )

% flowing inside the network through the nodes 1
% ,...,
% b
% , respectively.

\section{Statement and Discussion of the Characterization}
\label{sec:result}

%The following theorem characterizes the complex symmetric matrices
%$\Lambda$ that describe the response matrices of networks.
%response matrices

%.. response matrix

\begin{theorem}\label{main}
  Let $\Lambda = S+Ti$ be a $b\times b$ complex symmetric matrix, for $b\ge 2$.
Then
$\Lambda$~describes the response map %trix
of some connected alternating-current
  network $G$ with $b$ boundary nodes if and only if it satisfies the following conditions\textup:
  \begin{enumerate}
  \item $\Lambda$ has row sums $0$.
  \item The real part $S$ is positive semidefinite.
  \item % $S$ has rank $b-1$. (Or equivalently: 
The only solutions of $Sx=0$ are the
    constant vectors $x=(c,c,\ldots,c)^T$.
  \end{enumerate}
%Moreover, 
If $\Lambda$ is given, one can construct
a suitable network $G$ with %at most 
$2b-2$ nodes.
\end{theorem}
It has been shown by
Prasolov and Skopenkov 
 that these conditions on $\Lambda$ are necessary, see in particular
\cite[Lemma~5.2(5)]{Prasolov} for condition~2 and
\cite[Remark~5.3]{Prasolov} for condition~3, which depends on $G$
being connected.
For % the better-known 
direct-current networks, i.\,e., %that is, 
networks with
real (and nonnegative) conductances, it is known that the
response matrix must fulfill the above conditions 1--3, plus the condition
that the off-diagonal elements are $\le 0$. In this case, sufficiency
is trivial, since one can take $\Lambda$ directly as the Laplace %Kirchhoff
matrix
(see Section~\ref{sec:preliminaries})
of
a network, without any interior nodes.

For
alternating-current networks,
sufficiency of
conditions 1--3
 is easy for $b=2$, by the same reason as for
direct-current networks:
 Condition~1 implies that
 $\Lambda$ is of the form $
\left( \begin{smallmatrix}
   c&-c\\-c&c
 \end{smallmatrix}\right)
$,
and by
 condition~2,
 % the off-diagonal element
 $c$ % of $\Lambda$
must have positive % negative
real part.
% Thus, $\Lambda$ can be used directly as the Laplace
% matrix.
No interior nodes are
needed, and the network consists of a single edge of conductance~$c$.
For $b\ge3$, however, 
the matrix $\Lambda$ can have off-diagonal entries with positive
real part, and this implies that interior nodes are required,
see for example the $3\times3$ matrix~\eqref{example} in
Section~\ref{sec:example}.
For $b=3$,
sufficiency
has been established
by
Prasolov and Skopenkov 
\cite[Theorem~4.7]{Prasolov}, using one interior node.
Their construction is different from ours when specialized to the case $b=3$.
% One can
%show that an additional node may be necessary.
We do not know whether the number % bound
$b-2$ %on the
of interior %additional
nodes is optimal for $b\ge4$.

\section{The Laplace Matrix and the Response Matrix}
\label{sec:preliminaries}

We will now recall how the matrix of the response map is computed,
and we will prove a simple lemma that will be useful.
The statements of this section are basic linear algebra and hold
both over the
reals and over the complex numbers. %field.

In the rest of the paper,
$\unit n$ denotes the $n\times n$ unit matrix,
$\ones_{m\times n}$
 denotes the all-ones matrix of dimension
${m\times n}$,
and $e_n=\ones_{1\times n}$ denotes
 the all-ones vector of size $n$.

We can assume without loss of generality that the network has no loops:
 $c_{uu}=0$.
 The
 \emph{Laplace matrix}
(or \emph{Kirchhoff matrix}) $L$ of the network is defined as follows:
The off-diagonal edges $L_{uv}$ for $u\ne v$ are the negative
conductances:
\begin{displaymath}
  L_{uv}=
  \begin{cases}
    -c_{uv},& \text{if there is an edge between $u$ and $v$,} \\
    0, & \text{otherwise.}
  \end{cases}
\end{displaymath}
The diagonal elements $L_{uu}$ are chosen to make the row sums 0:
$$L_{uu}=\sum_{uw} c_{uw}$$

If there are no interior nodes,
the response matrix % $R$
 is equal to~$L$.
Otherwise, the response matrix can calculated from $L$ as
follows.
%  % We assume that % the network has $n$ nodes, and 
% the first nodes
%  $1,2,\ldots,b$ are the boundary nodes.
Assume that the nodes $1,2,\ldots,b$ are the boundary nodes, and
$b+1,\ldots,b+n$ are the interior nodes.
Partition $L$ into blocks accordingly:
\begin{equation}
  \label{Laplace}
L=
\begin{pmatrix}
  A&B\\
B^T&C
\end{pmatrix}
\end{equation}
with $A\in \mathbb{C}^{b\times b}$,
 $B\in \mathbb{C}^{b\times n}$,
and $C\in \mathbb{C}^{n\times n}$.
\begin{proposition}
\label{response}
Let $L$ be the Laplace matrix of a connected network $G$
with at least one interior node,
partitioned into blocks according to \eqref{Laplace}.
Then the submatrix $C$ is invertible, and
the response matrix $R$ is equal to
the Schur complement of $C$ in $L$\textup:
$$R=A-BC^{-1}B^T\eqno\qed$$
\end{proposition}
This well-known formula follows easily from writing the equations
\thetag{\ref{eq:equilibrium}--\ref{eq:current}} in block form
and substituting the solutions,
 see
  \cite[Theorem~2.3]{Curtis-Ingerman-Morrow} % page 121
or
  \cite[Lemma~3.8 and Theorem~3.9]{Curtis-Morrow00}. % page 43
% ... discuss when or why $C$ is invertible ... 

 \begin{lemma}\label{row-sums}
   Assume that $L$ is a %symmetric
   $(b+n)\times(b+n)$ matrix of the form \eqref{Laplace},
 $C$~is invertible, and the last $n$ row sums of $L$ are zero.
 Then
the row sums of the response matrix $R
=A-BC^{-1}B^T$
 are zero
if and only if
the first $b$ row sums of $L$ are zero.
 \end{lemma}
 \begin{proof}
%Let $e$ denote the all-ones vector of appropriate size.
By assumption, the last $n$ row sums of $L$ are zero: $B^Te_b+Ce_n=0$, which implies
$C^{-1}B^Te_b=-e_n$. In view of this, zero row sums of $R$ mean that $0=
Re_{b}
=Ae_b-BC^{-1}B^Te_b
=Ae_b+Be_n$, which in turn expresses the fact that % is equivalent to
the first $b$ row sums of $L$
are %being
zero.
\end{proof}

%  Then $\Lambda$ is the response matrix of some alternating-current
%  network with at most $2b-2$ nodes.

\section{Proof of Sufficiency and Construction of the Network}
\label{sec:sufficiency}

Before giving the proof, we will
study the simple example of just one interior node
$y$ in addition to the boundary nodes $x_1,\ldots,x_b$, see Figure~\ref{fig:example}.
We give the edge between $x_u$ and $y$ a conductance
$\delta +iw_u$ with a small positive real part $\delta$, leaving
the imaginary part $w_u$ as a parameter, subject to the constraint
$\sum_{u=1}^b w_u = 0$.
\begin{figure}[h]
  \centering
\includegraphics{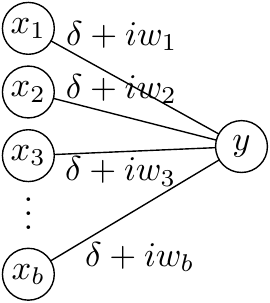}
  \caption{A network with one interior node $y$}
  \label{fig:example}
\end{figure}
Calculating the response matrix $R$ by Proposition~\ref{response} gives $\real
r_{uv}= (w_uw_v -\delta^2)/\delta b $ for the off-diagonal entries. Thus,
with this method, one can produce, for the real part of the response
matrix, any positive semidefinite % symmetric
rank-one matrix
$(w_uw_v)$ with row sums 0, up to a small error % term
$\delta/b$ 
in all entries. The parameters $w_u$ must be suitably scaled to
offset the factor $1/\delta b$.

By adding more interior nodes in this way, we can build up a sum of
positive semidefinite
 rank-one matrices, and hence an arbitrary
positive semidefinite
 matrix $S$ with row sums~0.
This is the main idea of the construction for the real part $S$ of $\Lambda$.
 We must take care of the accumulated error terms in % that are added to
the entries. We are able to accommodate them since is there is some 
tolerance for changing % increasing!
 all off-diagonal entries of $S$ by the same
amount while keeping the eigenvalues nonnegative. We will in fact
choose the parameter $\delta$ in such a way such that $S$ gets an
additional zero eigenvalue, and this will allow us to save one interior node
in the construction.

The complex part of $\Lambda$ can be handled as an afterthought.
We assign a %small
fixed positive real conductance
to every edge between two boundary nodes.
This gives us the freedom to adjust the complex part of these edges as we like,
and to achieve any desired complex part of the response matrix.

%\begin{proof}
We now begin with the formal proof.
As mentioned in Section~\ref{sec:result}, the case $b=2$ can be easily
handled without interior nodes. We
thus assume $b\ge 3$ in order to avoid degenerate situations.
  Since $S$ is symmetric, it can be written as
  \begin{equation*}
    S=U D U^T
  \end{equation*}
  with
a diagonal matrix $D=\diag(\lambda_1,\ldots,\lambda_b)$ whose entries are the
 eigenvalues $\lambda_1\le \lambda_2\le\cdots \le
\lambda_b$, and
  an orthogonal matrix $U$ whose columns are the corresponding normalized 
eigenvectors of $S$.
Since $S$ is positive semidefinite, all eigenvalues are nonnegative.
By assumption~3, $S$ has %rank $b-1$, and there is
only one zero eigenvalue:
$0=\lambda_1< \lambda_2$.
From assumption~3 (or~1) of the theorem, we know the eigenvector
corresponding to $\lambda_1=0$: it is the vector with all entries
equal.
Thus we can take
the vector $e_b %(1,\ldots,1)
/\sqrt b$
as the first
column of $U$.

We now decrease all positive eigenvectors by % some small amount
$\lambda_2$, so that they remain nonnegative. %positive, 
To achieve this, we replace the diagonal matrix of eigenvalues
$D$ by
\begin{equation*}
 D' =  D-\lambda_2 \left[\unit b - 
\left(  \begin{smallmatrix}
    1&0&0&\ldots\\
    0&0&0&\ldots\\
    0&0&0&\ldots\\
\cdot&\cdot&\cdot&\cdot\hfill\\[-0.65\baselineskip]
\cdot&\cdot&\cdot&\phantom\cdot\cdot\hfill\\[-0.65\baselineskip]
\cdot&\cdot&\cdot&\phantom{\cdot\cdot}\cdot\hfill
  \end{smallmatrix}\right)
\right],
\end{equation*}
and this results in the matrix
\begin{align}\label{S'}
  S'
&=U D' U^T
\\&\nonumber
=U D U^T - \lambda_2 U U^T + \lambda_2 e_b/\sqrt b \cdot e_b^T/\sqrt b
=S - \lambda_2 \unit b + \lambda_2 \ones_{b\times b}/b
,
\end{align}
%where $\ones $ is the all-ones matrix.
%where $\ones_{m\times n}$ %$\mathbf{1}$
% denotes an all-ones matrix of dimension
%${m\times n}$.
In other words, $S'$ is obtained from $S$ by increasing each off-diagonal entry by
$\lambda_2/b$ and adjusting the diagonal so that the row sums remain $0$.

It will be convenient to rewrite \eqref{S'} in a different way:
\begin{equation*}
  S'
=U
\sqrt{ D'}
\sqrt{ D'}
 U^T
=
\bigl(U\sqrt{ D'}
\bigr)
\bigl(U\sqrt{ D'}
\bigr)^T
 = V V ^T,
\end{equation*}
where the columns of $V=U\sqrt{ D'}
$ are no longer normalized.
 The columns of $V$ correspond to the interior nodes that we will
 add to the network.
 We can reduce their number % of columns of $V$
 by observing
that,
as the first two diagonal entries of $D'$ are zero,
the first two columns of $V$ %(and possibly more) 
are zero.
They contribute nothing to $S'$
 and can be omitted, resulting in the
 real $b\times (b-2) $ matrix $W%=(w_{kj})
 $ %, for some $n\le b-2$,
 with
%$S'=WW^T$ and 
$$W W^T=S' = S 
- \lambda_2 \unit b
+  \lambda_2 \ones_{b\times b} /b.
$$
To obey the conventions of Section~\ref{sec:preliminaries},
we denote by $n=b-2$ the number of columns of~$W$.
(If the eigenvalue $\lambda_2$ has higher multiplicity, then
more %than two
columns of $V$ are zero and $n$ could be
 reduced.) % below $b-2$.)
Since the columns of $U$ are orthogonal and its
first column is $e_b%(1,1,\ldots,1)
/\sqrt b$, the remaining columns of $U$, and
hence all columns of $W$, are orthogonal to~$e_b$: % this column:
%have
\begin{equation}\label{sum-0}
%\begin{equation*}
  W^T e_b = 0
%\end{equation*}
%or, more explicitly,
%  \sum_{k=1}^b w_{kj} = 0
\end{equation}
%for $j=1,\ldots,n$.

We are now ready to define the network.
The imaginary parts of the conductances of the edges between
the boundary nodes are represented by a
\pagebreak[3]
symmetric real $b\times b$ matrix~$F$
%, representing the imaginary part of
%the upper left corner, 
that will be determined later.
With the parameters
$\delta := %b
\lambda_2/2n$
and $\eps := \sqrt{b\delta}$, we
 set up the symmetric $(b+n)\times(b+n)$ matrix 
\begin{equation*}
  L:=
  \begin{pmatrix}
    \lambda_2 \unit b - \lambda_2/2b\cdot \ones_{b\times b} + Fi & 
-\delta \ones_{b\times  n }
 + \eps W i
\\
-\delta \ones_{ n \times b} + \eps W^T i
& \delta b \unit n
  \end{pmatrix}.
\end{equation*}
We will show that it yields the desired response matrix $\Lambda$, and
it is indeed the Laplace matrix of a network with $n$ interior nodes.
Let us calculate the response matrix $R$
by Proposition~\ref{response}:
\begin{align*}
R &=
   \lambda_2 \unit b
 - \lambda_2/2b\cdot \ones_{b\times b} + Fi 
-
(-\delta \ones_{b\times  n }
 + \eps W i)
(\delta b \unit n)^{-1}
(-\delta \ones_{ n \times b} + \eps W^T i)
\end{align*}
 Its real part is
\begin{align*}
  \real R 
& =
   \lambda_2 \unit b - \lambda_2/2b\cdot \ones_{b\times b}
-\tfrac 1{\delta b}(\delta^2 n 
\ones_{b\times b}
-\eps^2WW^T)\\
&=
   \lambda_2 \unit b - 
 \ones_{b\times b}
(\lambda_2/2b
+\delta n /b)
+WW^T\\
&=
   \lambda_2 \unit b - 
 \ones_{b\times b}
(\lambda_2/2b
+\lambda_2/2b)
+S-\lambda_2 \unit b + \ones_{b\times b}\cdot \lambda_2/b 
= S,
\end{align*}
as desired. 
Since we can choose $F$ arbitrarily, the imaginary part of $R$ can be adjusted to any desired value~$T$.
The straightforward calculation gives
the explicit formula
\begin{equation}
\nonumber %  \label{eq:F}
 F := T-
 \sqrt{\delta/b}\, \bigl(
W\ones_{n\times  b }+
\ones_{b\times  n }W^T\bigr)
.
\end{equation}
 Thus, we have achieved $R=\Lambda$.

To conclude the proof, 
we still have to show that $L$
is the Laplace matrix of a network whose conductances have positive
real parts:
(a) All off-diagonal elements of $L$, whenever they are nonzero, have %obviously
negative real parts, namely $-\lambda_2/2b$ or $-\delta$, and hence the corresponding conductances have
positive real parts.
(b)~Finally, we need to check that the row sums of $L$ are zero.  The sums
of the last $n$ rows are
$-\delta \ones_{ n \times b}e_b + \eps W^Te_b i +\delta b \unit ne_n =
-\delta be_n + 0 + \delta be_n = 0 $,
by applying~\eqref{sum-0} for the second term.
%
%Now,
Since the row sums of $R=\Lambda$ are 0 by assumption,
Lemma~\ref{row-sums} allows us to conclude without further calculation
that the first $b$ row sums of $L$ are also~0. \qed

\section{An Example}
\label{sec:example}
We have seen that the imaginary part
of $\Lambda$
is not an issue. Thus, for simplicity, we choose a
real matrix as an example:
\begin{equation}\label{example}
  \Lambda=
  \begin{pmatrix}
    2 & 1 & -3 \\
1 & 2 & -3 \\
-3 & -3 & 6
  \end{pmatrix}
\end{equation}
This % following
 matrix % $\Lambda$ 
has some positive % and negative
off-diagonal entries. Hence, it is not the response matrix of a
network without interior nodes, and it
 cannot be the response matrix of any direct-current
network whatsoever. %at all.

The eigenvalues of %this matrix %
$\Lambda$
 are
$\lambda_1=0, \lambda_2=1, \lambda_3=9$. The matrix 
$W$ has $n=1$ column, which is a scaled copy of % built from 
the eigenvector $(1,1,-2)^T$ that corresponds to $\lambda_3$.
One can recognize this vector in the last %row and
column of
the matrix $L$ below
in the imaginary parts. Our method sets
$\delta=1/2$, $\eps = \sqrt{3/2}$, and constructs the following
Laplace matrix:
%alternating-current network.
\begin{displaymath}
L =
\left(\begin{array}{ccc|c}
+ \frac{5}{6} -\frac{2}{3} i \, \sqrt{2} &
- \frac{1}{6} -\frac{2}{3} i \, \sqrt{2} &
- \frac{1}{6} +\frac{1}{3} i \, \sqrt{2} &
- \frac{1}{2} +i \, \sqrt{2} \\[\jot]
- \frac{1}{6} -\frac{2}{3} i \, \sqrt{2} &
+ \frac{5}{6} -\frac{2}{3} i \, \sqrt{2} &
- \frac{1}{6} +\frac{1}{3} i \, \sqrt{2} &
- \frac{1}{2} +i \, \sqrt{2} \\[\jot]
- \frac{1}{6} +\frac{1}{3} i \, \sqrt{2} &
- \frac{1}{6} +\frac{1}{3} i \, \sqrt{2} &
+ \frac{5}{6} +\frac{4}{3} i \, \sqrt{2} &
- \frac{1}{2} -2 i \, \sqrt{2} \\[\jot]
\hline
\omit\vbox to\jot{}&&&\\
- \frac{1}{2} +i \, \sqrt{2} &
- \frac{1}{2} +i \sqrt{2} &
- \frac{1}{2} -2 i \, \sqrt{2} &
 \frac{3}{2}
\end{array}\right)
\end{displaymath}

\bibliography{alt-current-networs}

% Let Lambda be the desired response matrix.

% 1. As in your proof, we can find a small delta>0 such that
% the matrix

%       B = Re Lambda + delta * (J-b*I)

% still satisfies the requirement of being positive semidefinite
% with (1,1,1,..1) as the only vector in the kernel.
% (J is the all-ones matrix, I is the identity matrix).

% 2. Write B = U^T D U where the b-1 columns of
% U are the nonzero normalized eigenvectors and
% D is the diagonal matrix of eigenvalues.

% 3. Choose eps = b*delta/(b-1)

% 4. For each eigenvector (c_1,...c_b) with eigenvalue v,
% add a new node and connect it to the terminals
% j=1,2,...,b by edges of conductance

%    eps + i*(c_u*sqrt(b*eps*v))

% 5. The resulting network has the correct real parts of
% Lambda. (The complex parts can be adjusted easily,
% by using part of the delta-perturbation for this purpose.
% This I did not carry out.)

\end{document}